\documentclass{article}
\pdfoutput=1
\usepackage[pdftex,
            pdfauthor={Alexis Marchand},
            pdftitle={Free representations of outer automorphism groups of free products via characteristic abelian coverings},
            colorlinks=true, linkcolor=black, urlcolor=black, citecolor=black]{hyperref}
\usepackage[utf8]{inputenc}
\usepackage[T1]{fontenc}
\usepackage[english]{babel}
\usepackage{lmodern}
\usepackage{verbatim}
\usepackage[dvipsnames]{xcolor}
\usepackage[nottoc]{tocbibind}
\usepackage{makeidx}
\usepackage{graphicx}
\usepackage{tikz}
\usepackage{amsmath}
\usepackage{amsthm}
\usepackage{amssymb}
\usepackage{mathrsfs}
\usepackage{stmaryrd}
\usepackage{eqnarray}
\usepackage{dsfont}
\usepackage{relsize}
\usepackage{extarrows}
\usepackage{caption}
\usepackage{subcaption}
\usepackage{enumitem}
\usepackage{authblk}
\usepackage{geometry}

\usetikzlibrary{decorations.markings,graphs,graphs.standard}

\theoremstyle{plain}
\newtheorem{prop}{Proposition}
\numberwithin{prop}{section}
\newtheorem*{thrm*}{Theorem}
\newtheorem{thrm}[prop]{Theorem}
\newtheorem{mthrm}{Theorem}
\newtheorem{cons}[prop]{Corollary}
\newtheorem{lmm}[prop]{Lemma}

\newtheorem{innermainthrms}{Theorems}
\newenvironment{mainthrms}[1]
  {\renewcommand\theinnermainthrms{#1}\innermainthrms}
  {\endinnermainthrms}

\newtheorem{innermaincons}{Corollary}
\newenvironment{maincons}[1]
  {\renewcommand\theinnermaincons{#1}\innermaincons}
  {\endinnermaincons}

\theoremstyle{definition}
\newtheorem*{rk}{Remark}
\newtheorem*{rks}{Remarks}
\newtheorem{defi}[prop]{Definition}
\newtheorem{nota}[prop]{Notation}

\newcommand{\longequal}{\ensuremath{=\joinrel=\joinrel=}}

\renewcommand{\leq}{\leqslant}
\renewcommand{\geq}{\geqslant}
\renewcommand{\subset}{\subseteq}

\DeclareMathOperator{\Stab}{Stab}
\DeclareMathOperator{\Ker}{Ker}
\DeclareMathOperator{\Imm}{Im}
\DeclareMathOperator{\id}{id}
\DeclareMathOperator{\ad}{ad}
\DeclareMathOperator{\he}{he}
\DeclareMathOperator{\HE}{HE}
\DeclareMathOperator{\fhe}{fhe}
\DeclareMathOperator{\FHE}{FHE}
\DeclareMathOperator{\Deck}{Deck}
\DeclareMathOperator{\Aut}{Aut}

\DeclareMathOperator{\Out}{Out}
\DeclareMathOperator{\Inn}{Inn}

\makeatletter
\newcommand{\address}[1]{\gdef\@address{#1}}
\newcommand{\email}[1]{\gdef\@email{\texttt{#1}}}
\newcommand{\@endstuff}{\par\vspace{\baselineskip}\noindent\small
\begin{flushleft}{\scshape\@address}\newline\textit{E-mail address:} \@email\end{flushleft}}
\AtEndDocument{\@endstuff}
\makeatother

\title{Free representations of outer automorphism groups of free products via characteristic abelian coverings}
\author{Alexis Marchand}
\address{Institut Fourier, UMR 5582, Laboratoire de Mathématiques,\newline Université Grenoble Alpes, CS 40700, 38058 Grenoble cedex 9, France}
\email{aptm3@cam.ac.uk}

\begin{document}

\maketitle

\begin{abstract}
    Given a free product $G$, we investigate the existence of faithful free representations of the outer automorphism group $\Out(G)$, or in other words of embeddings of $\Out(G)$ into $\Out\left(F_m\right)$ for some $m$. This is based on a work of Bridson and Vogtmann in which they construct embeddings of $\Out\left(F_n\right)$ into $\Out\left(F_m\right)$ for some values of $n$ and $m$ by interpreting $\Out\left(F_n\right)$ as the group of homotopy equivalences of a graph $X$ of genus $n$, and by lifting homotopy equivalences of $X$ to a characteristic abelian cover of genus $m$. Our construction for a free product $G$, using a presentation of $\Out(G)$ due to Fuchs-Rabinovich, is written as an algebraic proof, but it is directly inspired by Bridson and Vogtmann's topological method and can be interpreted as lifting homotopy equivalences of a graph of groups. For instance, we obtain a faithful free representation of $\Out(G)$ when $G=F_d\ast G_{d+1}\ast\cdots\ast G_n$, with $F_d$ free of rank $d$ and $G_i$ finite abelian of order coprime to $n-1$.
\end{abstract}

\section{Introduction}

    \paragraph{}
    Given a group $G$, denote by $\Aut(G)$ the group of automorphisms of $G$. Write $\Inn(G)$ for the normal subgroup of $\Aut(G)$ consisting of conjugations by elements of $G$ and consider the quotient\[\Out(G)=\Aut(G)/\Inn(G).\]The group $\Out(G)$ is called the \emph{outer automorphism group} of $G$.
    
    The study of $\Out\left(F_n\right)$, where $F_n$ is the free group of rank $n$, gained momentum following the work of Culler and Vogtmann \cite{cv} who introduced the concept of outer space, an analogue for graphs of the Teichmüller space of a surface, and on which $\Out\left(F_n\right)$ acts. It seems natural to try to generalise by studying $\Out(G)$ when $G$ is a free product. Let us mention in this direction the construction of an outer space associated with a free product by Guirardel and Levitt \cite{gl}.
    
    A possible strategy for studying $\Out(G)$ is to try to embed it into a better-understood group, for instance $\Out\left(F_n\right)$. Following the terminology coined by Kielak \cite{kielak}, and by analogy with linear representations, we will call such a morphism from a group $\Gamma$ to $\Out\left(F_n\right)$ a \emph{free representation} of $\Gamma$.
    
    We are led to hope that the existence of faithful free representations of a group could lead to extend properties of $\Out\left(F_n\right)$ to this group. For instance, we say that a group $\Gamma$ satisfies the \emph{Tits alternative} relative to a class $\mathcal{C}$ of groups if every subgroup of $\Gamma$ is either virtually in $\mathcal{C}$ or contains a non-abelian free subgroup. Bestvina, Feighn and Handel \cite{bfh,bfh2} proved that $\Out\left(F_n\right)$, and therefore any group admitting a faithful free representation, satisfies the Tits alternative relative to abelian groups. This is parallel to Tits' original result \cite{tits}, according to which any linear group (i.e. any group admitting a faithful linear representation) satisfies the Tits alternative relative to solvable groups.
    
    Furthermore, Collins \cite{collins} showed that, if $G$ is a free product of a finite number of finite groups, then there is a free finite index subgroup $F$ of $G$ such that $\Aut(G)/F$ embeds into $\Out(F)$. This enables Collins to use cohomological results of Culler and Vogtmann \cite{cv} on $\Out\left(F_n\right)$ to deduce that $\Out(G)$ is virtually torsion free and of finite virtual cohomological dimension.
    
    We therefore set ourselves the task of constructing faithful free representations of the group $\Out(G)$, where $G$ is a free product, in the hope that this could lead to generalisations of some structural results that are known for $\Out\left(F_n\right)$.
    
    \paragraph{}
    Towards this goal, we take inspiration from work of Bridson and Vogtmann \cite{bv} in which they explore the existence of embeddings $\Out\left(F_n\right)\hookrightarrow\Out\left(F_m\right)$ for different values of $n$ and $m$. Their method consists in viewing $F_n$ as the fundamental group of a graph $X$ of genus $n$ and in analysing certain characteristic abelian coverings $p:\hat{X}\rightarrow X$ with $\pi_1\hat{X}\cong F_m$.
    
    In this context, write $N=\pi_1\hat{X}\leq\pi_1X=F_n$. Consider the Galois group $\Deck(p)$ of the covering, the group $\HE(X)$ of homotopy equivalences of $X$ up to homotopy, and the group $\FHE(p)$ of fibre-preserving homotopy equivalences of $\hat{X}$ up to fibrewise homotopy (see \cite{bv} for precise definitions). Bridson and Vogtmann \cite[Prop. 1]{bv} prove that there is an equivalence of exact sequences:
    \begin{center}
        \begin{tikzpicture}[align=center, node distance = 1cm, auto]
            \node at (-1.5,0) (Y) {$1$};
            \node at (0,0) (Z) {$F_n/N$};
            \node at (2,0) (A) {$\Aut\left(F_n\right)/N$};
            \node at (4,0) (B) {$\Out\left(F_n\right)$};
            \node at (5.5,0) (C) {$1$};
            \draw [->] (Y) -> (Z);
            \draw [->] (Z) -> (A);
            \draw [->] (A) -> (B);
            \draw [->] (B) -> (C);
            \node [ below of=Y] (Yp) {$1$};
            \node [ below of=Z] (Zp) {$\Deck(p)$};
            \node [ below of=A] (Ap) {$\FHE(p)$};
            \node [ below of=B] (Bp) {$\HE\left(X\right)$};
            \node [ below of=C] (Cp) {$1$};
            \draw [->] (Yp) -> (Zp);
            \draw [->] (Zp) -> (Ap);
            \draw [->] (Ap) -> (Bp);
            \draw [->] (Bp) -> (Cp);
            \draw (0,-0.5) node[rotate=-90] {$\cong$};
            \draw (2,-0.5) node[rotate=-90] {$\cong$};
            \draw (4,-0.5) node[rotate=-90] {$\cong$};
        \end{tikzpicture}
    \end{center}
    
    From this observation, it follows that we can construct an embedding $\Out\left(F_n\right)\hookrightarrow\Aut\left(F_n\right)/N$ by splitting the projection $\FHE(p)\twoheadrightarrow\HE\left(X\right)$, or in more topological terms, by lifting homotopy equivalences of $X$ to the covering space $\hat{X}$. But since $N$ is characteristic in $F_n$, there is a restriction homomorphism $\Aut\left(F_n\right)/N\rightarrow\Out(N)$, which is injective in most cases, and therefore we can sometimes obtain an embedding $\Out\left(F_n\right)\hookrightarrow\Out(N)$ by composition. This allows Bridson and Vogtmann \cite[Cor. A]{bv} to show that, if $m=r^n(n-1)+1$, with $r$ coprime to $n-1$, then there is an embedding $\Out\left(F_n\right)\hookrightarrow\Out\left(F_m\right)$.
    
    \paragraph{}
    We aim to apply this method when the free group $F_n$ is replaced by a free product $G$. Our first idea was to generalise Bridson and Vogtmann's equivalence of exact sequences to the setting of graphs of groups. To do this, we changed viewpoints and focused on the action of groups on trees, following an idea that is used for instance by Guirardel and Levitt \cite{gl}. Our objects of interest were therefore trees equipped with group actions. Noting that in the case of graphs --- which correspond to trees with free actions --- notions of homotopy and coverings can be given purely combinatorial definitions, we worked with analogous definitions for general trees with actions and obtained an equivalence of exact sequences generalising Bridson and Vogtmann's observation. More details can be found in the author's Master's thesis \cite{marchand}.
    
    This work on homotopy and coverings of trees with group actions was very helpful for intuition: indeed, if phrased in this language, our construction of embeddings $\Out(G)\hookrightarrow\Out(N)$ is a direct generalisation of Bridson and Vogtmann's embeddings $\Out\left(F_n\right)\hookrightarrow\Out\left(F_m\right)$. However, it turns out that the actual proofs can be made shorter and in our opinion clearer in the algebraic setting, i.e. by splitting the upper exact sequence\[1\rightarrow G/N\rightarrow\Aut(G)/N\rightarrow\Out(G)\rightarrow1\]without refering to the topological lower one.
    
    To split the projection $\Aut(G)/N\twoheadrightarrow\Out(G)$, we use a presentation of $\Out(G)$ due to Fuchs-Rabinovich \cite{fr1,fr2}. Each element of the generating set $S$ of $\Out(G)$ is defined by specifying a representative in $\Aut(G)$. This gives rise to a map $S\rightarrow\Aut(G)/N$, but this map does not respect the relations of $\Out(G)$. Therefore, we correct the image of each generator and check that the relations are satisfied, which yields a group homomorphism $\Out(G)\rightarrow\Aut(G)/N$ which is a splitting of the projection. Our corrections are analogous to those of Bridson and Vogtmann in \cite{bv}, where they use a presentation of $\Out\left(F_n\right)$ to obtain a splitting of $\Aut\left(F_n\right)/N\twoheadrightarrow\Out\left(F_n\right)$.
    
    \paragraph{}
    We obtain the following results.
    
    \begin{mainthrms}{\ref{thrm_plong_outg} and \ref{thrm_plong_outg_2}}
        Let $G=F_d\ast G_{d+1}\ast\cdots\ast G_n$, with $n\geq 2$, $F_d$ free of rank $d$, $G_{d+1},\dots,G_n$ abelian and not isomorphic to $\mathbb{Z}$.
        \begin{enumerate}
            \item[\textnormal{(\ref{thrm_plong_outg})}] If $d=0$, choose an integer $r_i\in\mathbb{Z}$ coprime to $n-1$ for each $i\in\left\{1,\dots,n\right\}$, in such a way that $r_i=r_j$ as soon as $G_i\cong G_j$, and set $N=G'G_1^{r_1}\cdots G_n^{r_n}$.
            \item[\textnormal{(\ref{thrm_plong_outg_2})}] In general, choose an integer $r\in\mathbb{Z}$ coprime to $n-1$ and set $N=G'G^r$.
        \end{enumerate}
        In both cases, there is an embedding $\Out(G)\hookrightarrow\Out(N)$.
    \end{mainthrms}
    
    We are particularly interested in the cases where $N$ is free of finite rank, because they give us faithful free representations of $\Out(G)$. For example, we have the following two corollaries.
    
    \begin{maincons}{\ref{cons_rep_libre}}
        Let $G=F_d\ast G_{d+1}\ast\cdots\ast G_n$, with $n\geq2$, $F_d$ free of rank $d$, $G_{d+1},\dots,G_n$ finite abelian. We assume that $n-1$ is coprime to the order $\left|G_i\right|$ of each factor $G_i$. Then there is a free subgroup $F$ of finite rank and of finite index in $G$ such that there is an embedding\[\Out(G)\hookrightarrow\Out(F).\]In particular, $\Out(G)$ has a faithful free representation.
    \end{maincons}
    \begin{maincons}{\ref{cons_rep_libre_2}}
        Let $G=F_d\ast G_{d+1}\ast\cdots\ast G_n$, with $F_d$ free of rank $d$, $G_{d+1},\dots,G_n$ finite abelian. Then there is an integer $k\geq0$ such that $\Out\left(F_k\ast G\right)$ has a faithful free representation.
    \end{maincons}
    
    It is already known that $\Out(G)$ satisfies the Tits alternative relative to abelian groups whenever $G$ is a free product with the assumptions of Corollary \ref{cons_rep_libre}. In fact, to prove that a group $\Gamma$ satisfies the Tits alternative relative to some class $\mathcal{C}$ of groups, it suffices to construct a \emph{virtual embedding} of $\Gamma$ into another group $\Delta$ satisfying the Tits alternative relative to $\mathcal{C}$, i.e. an embedding of a finite index subgroup of $\Gamma$ into $\Delta$. In particular, if $G=F_d\ast G_{d+1}\ast\cdots\ast G_n$ with $G_i$ finite (not necessarily abelian), then we can use a result of Carette \cite[Theorem A]{carette} to deduce that the projection $p:\Aut(G)\twoheadrightarrow\Out(G)$ splits virtually, i.e. there is a finite index subgroup $H$ of $\Out(G)$ and an embedding $i:H\hookrightarrow\Aut(G)$ such that $p\circ i=\id_H$. This implies in particular that $\Out(G)$ virtually embeds into $\Aut(G)$ and that its image intersects $\Inn(G)$ trivially. Now, by Lemma \ref{lmm_collins} of this text, if we pick a free characteristic finite index subgroup $F$ of $G$, then the restriction map $\Aut(G)\rightarrow\Aut(F)$ is injective, from which it follows by composition that $\Out(G)$ virtually embeds into $\Out(F)$. Since $\Out\left(F\right)$ satisfies the Tits alternative relative to abelian groups by the Bestvina-Feighn-Handel Theorem \cite{bfh,bfh2}, so does $\Out(G)$.
    
    However, our results give an alternative proof: under the assumptions of Corollary \ref{cons_rep_libre}, $\Out(G)$ has a faithful free representation, so it inherits the Tits alternative from $\Out\left(F_n\right)$.
    
    \paragraph{}
    Using the above results, we go on to examine the special case of $\Out\left(W_n\right)$ in Section \ref{sec_univ_coxeter}, where $W_n$ is the \emph{universal Coxeter group} of rank $n$, i.e. the free product of $n$ copies of $\mathbb{Z}/2$. If $n$ is even, then Corollary \ref{cons_rep_libre} applies and $\Out\left(W_n\right)$ has a faithful free representation. Looking at the construction more closely, we show that there is an embedding $\Out\left(W_n\right)\hookrightarrow\Out\left(F_m\right)$ with $m=2^{n-1}(n-2)+1$. Computing an explicit value of $m$ such that $\Out(G)$ embeds into $\Out\left(F_m\right)$ can actually be done with the same method for all free products $G$ satisfying the assumptions of Corollary \ref{cons_rep_libre}, but we limit ourselves to the special case of $W_n$ where the resulting formula is simple enough. We also use the example of $\Out\left(W_n\right)$ to check that our method of construction of free representations is not exhaustive: for instance, when $n=3$, $\Out\left(W_3\right)$ admits a faithful free representation not given by Corollary \ref{cons_rep_libre}. We expect $\Out(G)$ to have faithful free representations for many other free products $G$, as our method is probably a very special trick that we can use to construct free representations in specific cases.

    \paragraph{Structure of the paper.}
    We start by recalling some definitions and stating Grushko's Theorem in Section \ref{sect_not}. We then proceed to find some interesting characteristic subgroups $N$ of free products $G$ in Section \ref{sect_char_ab_cover}, and we describe the corresponding coverings of graphs of groups. Section \ref{sect_constr_embed_1} is devoted to the construction of embeddings $\Out(G)\hookrightarrow\Out(N)$ in the case where $G$ has no free factor; this is Theorem \ref{thrm_plong_outg}. The general case is similar except that the presentation of $\Out(G)$ that we use becomes more complicated; this is explained in Section \ref{sect_second_emb_result} and leads to Theorem \ref{thrm_plong_outg_2}. Finally, Section \ref{sec_univ_coxeter} discusses the special case of universal Coxeter groups.
    
    \paragraph{Acknowledgements.}
    This work is based on the author's Master's thesis, which was written in Institut Fourier, Grenoble. The author wishes to thank François Dahmani for suggesting this topic, supervising his thesis and helping with the preparation of this paper. The author is thankful to the anonymous referees for various helpful comments and for pointing out work of Carette and its application to the Tits alternative for $\Out(G)$ in a more general setting than that of this paper. Financial support from an ENS Lyon studentship is gratefully acknowledged.
    
\section{Notations and preliminaries}\label{sect_not}
    
    \paragraph{Conjugation action and outer automorphisms.}
    Given a group $G$ and an element $\gamma\in G$, there is an automorphism $\ad(\gamma)\in\Aut(G)$ defined by $\ad(\gamma):g\mapsto\gamma g\gamma^{-1}$. The map $\gamma\mapsto\ad(\gamma)$ defines a group homomorphism $\ad:G\rightarrow\Aut(G)$. This homomorphism is injective if $G$ has trivial centre. The group of \emph{inner automorphisms} of $G$ is the normal subgroup $\Inn(G)=\ad(G)\trianglelefteq\Aut(G)$ and the group of \emph{outer automorphisms} is the quotient $\Out(G)=\Aut(G)/\Inn(G)$.
    
    A subgroup $N$ of $G$ is \emph{characteristic} if it is preserved by all automorphisms of $G$.
    
    When we write an arrow $G\rightarrow\Aut(G)$ without further precision, this will always refer to the morphism $\ad$.
    
    \paragraph{Grushko's Theorem.}
    The fundamental theorem for the sequel is a rigidity result for free products. This result if often combined with the existence of free product decompositions and stated as follows.
    \begin{thrm*}[Grushko]
        If $G$ is a finitely generated group, then there are freely indecomposable groups $G_1,\dots,G_m$ not isomorphic to $\mathbb{Z}$, and a free group $F_r$ of rank $r\geq 0$ such that\[G=G_1\ast\cdots\ast G_m\ast F_r.\]If we have another decomposition $G=H_1\ast\cdots\ast H_n\ast F_s$ with $H_1,\dots,H_n$ freely indecomposable not isomorphic to $\mathbb{Z}$ and $F_s$ free of rank $s$, then $m=n$, $r=s$, and there is a permutation $\sigma\in\mathfrak{S}_m$ and elements $\gamma_1,\dots,\gamma_m\in G$ such that $H_i=\gamma_iG_{\sigma(i)}\gamma_i^{-1}$ for all $i$.
    \end{thrm*}
    In fact, we will only need the uniqueness of the decomposition. It seems that Kurosh should be credited with the uniqueness result, which is named \emph{Isomorphiesatz} in \cite{kurosh}. It can be derived as a corollary of Kurosh's \emph{Untergruppensatz} (see \cite{kurosh}), at least in the case where $G$ has no free factor. The original argument is combinatorial, but the \emph{Untergruppensatz} can be proved simply with the help of Bass-Serre Theory (see Serre's book \cite[\S{} I.5.5]{serre}).
    
    Grushko's contribution \cite{grushko} was to prove that the rank is additive for the operation of free product. The existence of the decomposition follows easily.

\section{Characteristic abelian covers associated with a free product}\label{sect_char_ab_cover}
    
    \paragraph{}
    Our goal is to apply Bridson and Vogtmann's method \cite{bv} to construct embeddings of $\Out(G)$, where $G$ is a free product. We first limit ourselves to the case where $G$ is a free product of groups that are freely indecomposable and not isomorphic to $\mathbb{Z}$ (i.e. there is no free factor in the Grushko decomposition of $G$). Moreover, we will need to assume that the factors are abelian. In Section \ref{sect_second_emb_result}, we will see how to deal with the case where some of the factors are isomorphic to $\mathbb{Z}$ (i.e. where $G$ has a free factor).
    
    The characteristic cover $\hat{X}$ of the graph $X$ chosen by Bridson and Vogtmann to construct embeddings of $\Out\left(F_n\right)$ is a finite abelian cover, i.e. a normal cover whose Galois group is finite abelian. A good reason to make this choice is the simplicity of the geometry of $\hat{X}$: it is the Cayley graph of $\left(\mathbb{Z}/r\right)^n$ with its standard generating set, which can be seen as a grid embedded in an $n$-dimensional torus.
    
    Analogously, we want to find an interesting finite characteristic abelian cover corresponding to a free product, or in other words a finite index characteristic subgroup $N$ of $G$ such that $G/N$ is abelian.
    
    The following lemma provides us with such subgroups.
    
    \begin{lmm}\label{lmm_gpg1r1gnrn_caract}
        Let $G=G_1\ast\cdots\ast G_n$ be a free product of groups that are freely indecomposable and not isomorphic to $\mathbb{Z}$. Choose integers $r_1,\dots,r_n\in\mathbb{Z}$ such that $r_i=r_j$ as soon as $G_i\cong G_j$. Then the subgroup $N=G'G_1^{r_1}\cdots G_n^{r_n}$ is characteristic in $G$.
    \end{lmm}
    \begin{proof}
        Let $\varphi\in\Aut(G)$. By Grushko's Theorem, there is a permutation $\sigma\in\mathfrak{S}_n$ and elements $\gamma_1,\dots,\gamma_n\in G$ such that\[\varphi\left(G_i\right)=\gamma_iG_{\sigma(i)}\gamma_i^{-1}.\]Note that $r_i=r_{\sigma(i)}$ for all $i$ because $G_i\cong G_{\sigma(i)}$. We have\[\varphi\left(G'\right)\subset G'\subset N\]because $G'$ is characteristic in $G$. Moreover,\[\varphi\left(G_i^{r_i}\right)=\gamma_iG_{\sigma(i)}^{r_i}\gamma_i^{-1}\subset\left[\gamma_i,G_{\sigma(i)}^{r_i}\right]G_{\sigma(i)}^{r_i}\subset G'G_{\sigma(i)}^{r_{\sigma(i)}}\subset N.\]Hence, $\varphi(N)\subset\varphi\left(G'\right)\varphi\left(G_1^{r_1}\right)\cdots\varphi\left(G_n^{r_n}\right)\subset N$.
    \end{proof}
    
    \paragraph{}
    The case where the factors $G_i$ are abelian is particularly nice because of the following lemma.
    
    \begin{lmm}\label{lmm_sousgp_deriv_libre}
        Let $G=G_1\ast\cdots\ast G_n$ be a free product of abelian groups. If $A$ is a Bass-Serre tree of the free product, then the commutator subgroup $G'$ acts freely on $A$. In particular, $G'$ is free.
    \end{lmm}
    \begin{proof}
        If $v$ is a vertex of $A$, then its stabiliser under $G$ is given by\[\Stab_G(v)=\gamma G_i\gamma^{-1},\]with $\gamma\in G$ and $i\in\left\{1,\dots,n\right\}$. Since the abelianisation of $G$ is the direct product $G_1\times\cdots\times G_n$, the stabiliser of $v$ under $G'$ is given by\[\Stab_{G'}(v)=\Stab_G(v)\cap G'=\gamma G_i\gamma^{-1}\cap\Ker\left(G\rightarrow G_1\times\cdots\times G_n\right)=1.\]It follows that $G'$ acts freely on $A$, so $G'$ is free.
    \end{proof}
    
    \paragraph{}
    Hence, if the factors are abelian, then the derived subgroup $G'$ is free. For the sake of simplicity, let us assume for the moment that the chosen characteristic subgroup $N=G'G_1^{r_1}\cdots G_n^{r_n}$ is actually equal to $G'$ by taking $G_i$ finite abelian and $r_i=\left|G_i\right|$.
    
    \begin{figure}[hbt]
        \centering
        \begin{subfigure}[t]{15em}
            \centering
            \begin{tikzpicture}[every node/.style={draw,circle,fill,inner sep=1pt},scale=0.6]
                \node[fill=none,draw=none] at (5,7) {$\hat{X}_1=G_1'\backslash A_1$};
                \node[fill=none,draw=none] at (3,-4.5) {$\mathbb{X}_1=G_1\backslash A_1$};
                
                \draw node[circle,fill=white] (p0) at (5,2.625) {};
                \draw node[circle,fill=white] (p1) at (5,4.25) {};
                \draw node[circle,fill=white] (p2) at (5,5.875) {};
                \draw node[circle,fill=white] (q) at (2.8,5) {};
                \draw node[circle,fill=white] (r) at (5.8,5) {};
                \draw node (x0) at (3.5,2.625) {};
                \draw node (x1) at (3.5,4.25) {};
                \draw node (x2) at (3.5,5.875) {};
                \draw node (y0) at (6.5,2.625) {};
                \draw node (y1) at (6.5,4.25) {};
                \draw node (y2) at (6.5,5.875) {};

                \draw [thick, color=ForestGreen] (p0) -- (x0);
                \draw [thick, color=ForestGreen] (p1) -- (x1);
                \draw [thick, color=ForestGreen] (p2) -- (x2);
                \draw [thick, color=ForestGreen] (p0) -- (y0);
                \draw [thick, color=ForestGreen] (p1) -- (y1);
                \draw [thick, color=ForestGreen] (p2) -- (y2);

                \draw [thick, color=Red] (q) -- (x0);
                \draw [thick, color=Red] (q) -- (x1);
                \draw [thick, color=Red] (q) -- (x2);
                \draw [thick, color=Red] (r) -- (y0);
                \draw [thick, color=Red] (r) -- (y1);
                \draw [thick, color=Red] (r) -- (y2);
                
                \coordinate (f1) at (5,1);
                \coordinate (f2) at (5,-0.5);
                \draw [->] (f1) -- (f2);
                
                \draw node[label=right:{\small$v_0$}] (bn) at (5,-4.5) {};
                \draw node[circle,fill=white, label=above:{\small$\mathbb{Z}/3$}] (brouge) at (5,-3.5) {};
                \draw node[circle,fill=white, label=below:{\small$\mathbb{Z}/2$}] (bvert) at (5,-5.5) {};
                \draw [thick, color=Red] (bn) -- (brouge);
                \draw [thick, color=ForestGreen] (bn) -- (bvert);
            \end{tikzpicture}
            \caption{$G_1=\mathbb{Z}/2\ast\mathbb{Z}/3$}
        \end{subfigure}
        \quad
        \begin{subfigure}[t]{20em}
            \centering
            \begin{tikzpicture}[every node/.style={draw,circle,fill,inner sep=1pt},scale=0.6]
                \node[fill=none,draw=none] at (5,8.5) {$\hat{X}_2=G_2'\backslash A_2$};
                \node[fill=none,draw=none] at (3.1,-4.5) {$\mathbb{X}_2=G_2\backslash A_2$};
            
                \draw node (p000) at (0,0) {};
                \draw node (p100) at (2.5,0) {};
                \draw node (p200) at (5,0) {};
                \draw node (p300) at (7.5,0) {};
                \draw node (p010) at (1,1) {};
                \draw node (p110) at (3.5,1) {};
                \draw node (p210) at (6,1) {};
                \draw node (p310) at (8.5,1) {};
                \draw node (p020) at (2,2) {};
                \draw node (p120) at (4.5,2) {};
                \draw node (p220) at (7,2) {};
                \draw node (p320) at (9.5,2) {};
                \draw node (p001) at (0,5) {};
                \draw node (p101) at (2.5,5) {};
                \draw node (p201) at (5,5) {};
                \draw node (p301) at (7.5,5) {};
                \draw node (p011) at (1,6) {};
                \draw node (p111) at (3.5,6) {};
                \draw node (p211) at (6,6) {};
                \draw node (p311) at (8.5,6) {};
                \draw node (p021) at (2,7) {};
                \draw node (p121) at (4.5,7) {};
                \draw node (p221) at (7,7) {};
                \draw node (p321) at (9.5,7) {};
                
                \draw node[circle,fill=white] (qvert00) at (0,2.5) {};
                \draw [thick, color=ForestGreen] (qvert00) -- (p001);
                \draw node[circle,fill=white] (qvert10) at (2.5,2.5) {};
                \draw [thick, color=ForestGreen] (qvert10) -- (p101);
                \draw node[circle,fill=white] (qvert20) at (5,2.5) {};
                \draw [thick, color=ForestGreen] (qvert20) -- (p201);
                \draw node[circle,fill=white] (qvert30) at (7.5,2.5) {};
                \draw [thick, color=ForestGreen] (qvert30) -- (p301);
                
                \draw node[circle,fill=white] (qvert01) at (1,3.5) {};
                \draw [thick, color=ForestGreen] (qvert01) -- (p011);
                \draw node[circle,fill=white] (qvert11) at (3.5,3.5) {};
                \draw [thick, color=ForestGreen] (qvert11) -- (p111);
                \draw node[circle,fill=white] (qvert21) at (6,3.5) {};
                \draw [thick, color=ForestGreen] (qvert21) -- (p211);
                \draw node[circle,fill=white] (qvert31) at (8.5,3.5) {};
                \draw [thick, color=ForestGreen] (qvert31) -- (p311);
                
                \draw node[circle,fill=white] (qvert02) at (2,4.5) {};
                \draw [thick, color=ForestGreen] (qvert02) -- (p021);
                \draw node[circle,fill=white] (qvert12) at (4.5,4.5) {};
                \draw [thick, color=ForestGreen] (qvert12) -- (p121);
                \draw node[circle,fill=white] (qvert22) at (7,4.5) {};
                \draw [thick, color=ForestGreen] (qvert22) -- (p221);
                \draw node[circle,fill=white] (qvert32) at (9.5,4.5) {};
                \draw [thick, color=ForestGreen] (qvert32) -- (p321);
                
                \draw node[circle,fill=white] (qbleu00) at (3.75,0.5) {};
                \draw [thick, color=Blue] (qbleu00) -- (p000);
                \draw [thick, color=Blue] (qbleu00) -- (p100);
                \draw [thick, color=Blue] (qbleu00) -- (p200);
                \draw [thick, color=Blue] (qbleu00) -- (p300);
                \draw node[circle,fill=white] (qbleu10) at (4.75,1.5) {};
                \draw [thick, color=Blue] (qbleu10) -- (p010);
                \draw [thick, color=Blue] (qbleu10) -- (p110);
                \draw [thick, color=Blue] (qbleu10) -- (p210);
                \draw [thick, color=Blue] (qbleu10) -- (p310);
                \draw node[circle,fill=white] (qbleu20) at (5.75,2.5) {};
                \draw [thick, color=Blue] (qbleu20) -- (p020);
                \draw [thick, color=Blue] (qbleu20) -- (p120);
                \draw [thick, color=Blue] (qbleu20) -- (p220);
                \draw [thick, color=Blue] (qbleu20) -- (p320);
                
                \draw node[circle,fill=white] (qbleu01) at (3.75,5.5) {};
                \draw [thick, color=Blue] (qbleu01) -- (p001);
                \draw [thick, color=Blue] (qbleu01) -- (p101);
                \draw [thick, color=Blue] (qbleu01) -- (p201);
                \draw [thick, color=Blue] (qbleu01) -- (p301);
                \draw node[circle,fill=white] (qbleu11) at (4.75,6.5) {};
                \draw [thick, color=Blue] (qbleu11) -- (p011);
                \draw [thick, color=Blue] (qbleu11) -- (p111);
                \draw [thick, color=Blue] (qbleu11) -- (p211);
                \draw [thick, color=Blue] (qbleu11) -- (p311);
                \draw node[circle,fill=white] (qbleu21) at (5.75,7.5) {};
                \draw [thick, color=Blue] (qbleu21) -- (p021);
                \draw [thick, color=Blue] (qbleu21) -- (p121);
                \draw [thick, color=Blue] (qbleu21) -- (p221);
                \draw [thick, color=Blue] (qbleu21) -- (p321);
                
                \draw node[circle,fill=white] (qrouge00) at (0.5,1.5) {};
                \draw [thick, color=Red] (qrouge00) -- (p000);
                \draw [thick, color=Red] (qrouge00) -- (p010);
                \draw [thick, color=Red] (qrouge00) -- (p020);
                \draw node[circle,fill=white] (qrouge10) at (3,1.5) {};
                \draw [thick, color=Red] (qrouge10) -- (p100);
                \draw [thick, color=Red] (qrouge10) -- (p110);
                \draw [thick, color=Red] (qrouge10) -- (p120);
                \draw node[circle,fill=white] (qrouge20) at (5.5,1.5) {};
                \draw [thick, color=Red] (qrouge20) -- (p200);
                \draw [thick, color=Red] (qrouge20) -- (p210);
                \draw [thick, color=Red] (qrouge20) -- (p220);
                \draw node[circle,fill=white] (qrouge30) at (8,1.5) {};
                \draw [thick, color=Red] (qrouge30) -- (p300);
                \draw [thick, color=Red] (qrouge30) -- (p310);
                \draw [thick, color=Red] (qrouge30) -- (p320);
                
                \draw node[circle,fill=white] (qrouge01) at (0.5,6.5) {};
                \draw [thick, color=Red] (qrouge01) -- (p001);
                \draw [thick, color=Red] (qrouge01) -- (p011);
                \draw [thick, color=Red] (qrouge01) -- (p021);            
                \draw node[circle,fill=white] (qrouge11) at (3,6.5) {};
                \draw [thick, color=Red] (qrouge11) -- (p101);
                \draw [thick, color=Red] (qrouge11) -- (p111);
                \draw [thick, color=Red] (qrouge11) -- (p121);            
                \draw node[circle,fill=white] (qrouge21) at (5.5,6.5) {};
                \draw [thick, color=Red] (qrouge21) -- (p201);
                \draw [thick, color=Red] (qrouge21) -- (p211);
                \draw [thick, color=Red] (qrouge21) -- (p221);            
                \draw node[circle,fill=white] (qrouge31) at (8,6.5) {};
                \draw [thick, color=Red] (qrouge31) -- (p301);
                \draw [thick, color=Red] (qrouge31) -- (p311);
                \draw [thick, color=Red] (qrouge31) -- (p321);
                
                \draw [thick, color=ForestGreen] (qvert00) -- (p000);
                \draw [thick, color=ForestGreen] (qvert10) -- (p100);
                \draw [thick, color=ForestGreen] (qvert20) -- (p200);
                \draw [thick, color=ForestGreen] (qvert30) -- (p300);
                \draw [thick, color=ForestGreen] (qvert01) -- (p010);
                \draw [thick, color=ForestGreen] (qvert11) -- (p110);
                \draw [thick, color=ForestGreen] (qvert21) -- (p210);
                \draw [thick, color=ForestGreen] (qvert31) -- (p310);
                \draw [thick, color=ForestGreen] (qvert02) -- (p020);
                \draw [thick, color=ForestGreen] (qvert12) -- (p120);
                \draw [thick, color=ForestGreen] (qvert22) -- (p220);
                \draw [thick, color=ForestGreen] (qvert32) -- (p320);
                
                \coordinate (f1) at (5,-1);
                \coordinate (f2) at (5,-2.5);
                \draw [->] (f1) -- (f2);
                
                \draw node[label=above:{\small$v_0$}] (bn) at (5,-4) {};
                \draw node[circle,fill=white, label=below:{\small$\mathbb{Z}/4$}] (bbleu) at (5,-5.5) {};
                \draw node[circle,fill=white, label=above right:{\small$\mathbb{Z}/3$}] (brouge) at (6.3,-3.25) {};
                \draw node[circle,fill=white, label=above left:{\small$\mathbb{Z}/2$}] (bvert) at (3.7,-3.25) {};
                \draw [thick, color=Red] (bn) -- (brouge);
                \draw [thick, color=Blue] (bn) -- (bbleu);
                \draw [thick, color=ForestGreen] (bn) -- (bvert);
            \end{tikzpicture}
            \caption{$G_2=\mathbb{Z}/2\ast\mathbb{Z}/3\ast\mathbb{Z}/4$}
        \end{subfigure}
        \caption{Examples of maximal abelian coverings associated with free products}\label{fig_rev_ab}
    \end{figure}
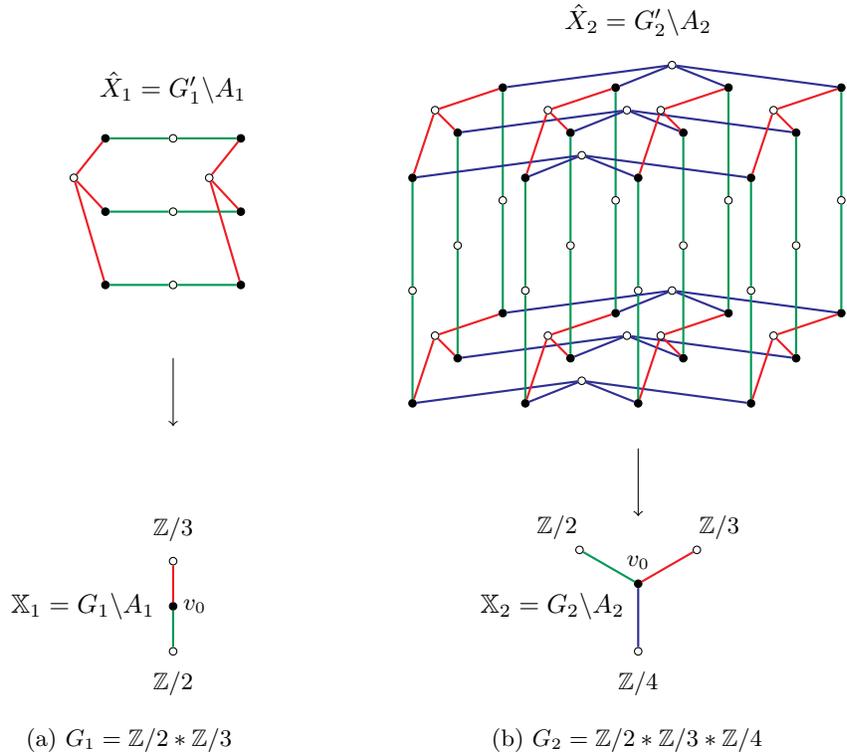
    
    Let $A$ be the Bass-Serre tree of $G$ corresponding to a star-shaped tree of groups $\mathbb{X}$ with $n$ vertices labelled by $G_1,\dots,G_n$, all linked to a common vertex $v_0$ with trivial stabiliser (see Figure \ref{fig_rev_ab}). The quotient graph $G'\backslash A$ can be seen as a characteristic abelian covering of $\mathbb{X}$.
    
    It is worthwhile thinking for a moment about the geometry of $\hat{X}=G'\backslash A$. Indeed, Bridson and Vogtmann's construction in \cite{bv} is done from the topological viewpoint, and it was this viewpoint that gave us the intuition to generalise their method --- even though our proofs are ultimately written in purely algebraic terms.
    
    To visualise the graph $\hat{X}$, one may first focus on the fibre of the unlabelled vertex $v_0$ of $\mathbb{X}=G\backslash A$ (in black on Figure \ref{fig_rev_ab}). The group $\Deck(p)\cong G/G'\cong G_1\times\cdots\times G_n$ acts simply transitively on this fibre, so the latter is in bijection with $G_1\times\cdots\times G_n$. Since $G'$ acts freely on $A$, we understand the local geometry of $\hat{X}=G'\backslash A$: each vertex of the fibre of $v_0$ is the endpoint of one edge for each factor $G_i$. The other endpoint of such an edge is in the fibre of the vertex labelled by $G_i$, which has degree $\left|G_i\right|$. We know in addition that, under the action of $\Deck(p)\cong G_1\times\cdots\times G_n$ on $\hat{X}$, the factor $G_i\leq\Deck(p)$ stabilises the whole fibre of the vertex of $\mathbb{X}$ labelled by $G_i$; this allows one to understand the action of $\Deck(p)$ on $\hat{X}$. We give two examples in Figure \ref{fig_rev_ab}.
    
    This picture of characteristic abelian covers $\hat{X}\rightarrow\mathbb{X}$ may be helpful for intuition and to understand the similarity between our construction and that of Bridson and Vogtmann. The interested reader may want to think about the topological interpretation of our proof: it provides a way to lift homotopy equivalences of $\mathbb{X}$ to $\hat{X}$.
    
\section{Proof of Theorem \ref{thrm_plong_outg}}\label{sect_constr_embed_1}
    
    \subsection{Generators and relations of \texorpdfstring{$\Out(G)$}{Out(G)}}\label{subsec_gen_rel_outg}
    
    \paragraph{}
    To split the projection $\Aut(G)/N\twoheadrightarrow\Out(G)$, we will use a presentation of $\Out(G)$, as in Bridson and Vogtmann's paper \cite{bv}. We use the work of Fuchs-Rabinovich \cite{fr1}, who gives a presentation of $\Aut(G)$ when $G$ is a free product. It will then suffice to add relations corresponding to inner automorphisms to get a presentation of $\Out(G)$.
    
    If $G$ has no free factor, Fuchs-Rabinovich \cite{fr1} defines three kinds of generators for $\Aut(G)$.
    
    \begin{nota}\label{nota_gen_fr}
        Let $G=G_1\ast\cdots\ast G_n$ be a free product of groups that are freely indecomposable and not isomorphic to $\mathbb{Z}$.
        \begin{description}
            \item[\textnormal{Factor automorphisms:}] For $i\in\left\{1,\dots,n\right\}$, and $\bar{\varphi}_i\in\Aut\left(G_i\right)$, we set
            \begin{align*}
                \varphi_i:g_i\in G_i&\mapsto\bar{\varphi}_i\left(g_i\right),\\
                g_j\in G_j&\mapsto g_j\quad\textrm{if $j\neq i$}.
            \end{align*}
            We write $\Phi_i=\left\{\varphi_i,\:\bar{\varphi}_i\in\Aut\left(G_i\right)\right\}$. Elements of $\Phi_i$ are called \emph{automorphisms of the factor $G_i$} and we write $\Phi=\bigcup_{i}\Phi_i$.
            \item[\textnormal{Permutation automorphisms:}] Given $i,j\in\left\{1,\dots,n\right\}$ such that $G_i\cong G_j$, we pick an isomorphism $\bar{\omega}_{ij}:G_i\rightarrow G_j$. We assume that choices have been made in such a way that $\bar{\omega}_{ii}=\id_{G_i}$ and $\bar{\omega}_{jk}\bar{\omega}_{ij}=\bar{\omega}_{ik}$. If $i\neq j$, we consider
            \begin{align*}
                \omega_{ij}:g_i\in G_i&\mapsto\bar{\omega}_{ij}\left(g_i\right),\\
                g_j\in G_j&\mapsto\bar{\omega}_{ji}\left(g_j\right),\\
                g_k\in G_k&\mapsto g_k\quad\textrm{if $k\neq i,j$}.
            \end{align*}
            Note that $\omega_{ij}^2=\id_G$. We set $\Omega=\left\{\omega_{ij},\:i\neq j,\:G_i\cong G_j\right\}$. Elements of $\Omega$ are called \emph{permutation automorphisms}.
            \item[\textnormal{Dehn twists:}] Given $i\neq j\in\left\{1,\dots,n\right\}$ and $\gamma_i\in G_i$, we define
            \begin{align*}
                \alpha_{ij}^{\left(\gamma_i\right)}:g_j\in G_j&\mapsto\gamma_i^{-1}g_j\gamma_i,\\
                g_k\in G_k&\mapsto g_k\quad\textrm{if $k\neq j$}.
            \end{align*}
            We set $A_{ij}=\left\{\alpha_{ij}^{\left(\gamma_i\right)},\:\gamma_i\in G_i\right\}$ and $A=\bigcup_{i\neq j}A_{ij}$. Elements of $A$ are called \emph{Dehn twists}.
        \end{description}
    \end{nota}
    
    A system of relations for $\Aut(G)$ is given by the following theorem. Observe that we consider the groups $\Aut\left(G_i\right)\cong\Phi_i$ and $G_i\cong A_{ij}$ as given.
    
    \begin{thrm}[Fuchs-Rabinovich \cite{fr1}]\label{thrm_pres_autg}
        Let $G=G_1\ast\cdots\ast G_n$ be a free product of groups that are freely indecomposable and not isomorphic to $\mathbb{Z}$. Then $\Aut(G)$ is generated by factor automorphisms, permutation automorphisms and Dehn twists. Moreover, the following relations give a presentation of $\Aut(G)$:
        \begin{enumerate}[label=\textnormal{(\arabic*)}]
            \item $\varphi_i\varphi_j=\varphi_j\varphi_i$ for $\varphi_i\in\Phi_i$, $\varphi_j\in\Phi_j$, $i\neq j$.\label{thrm_pres_autg_r1}
            \item $\varphi_i\alpha_{jk}^{\left(\gamma_j\right)}=\alpha_{jk}^{\left(\gamma_j\right)}\varphi_i$ for $\varphi_i\in\Phi_i$, $\gamma_j\in G_j$, $i\neq j$.\label{thrm_pres_autg_r2}
            \item $\varphi_i\alpha_{ik}^{\left(\gamma_i\right)}=\alpha_{ik}^{\left(\varphi_i\gamma_i\right)}\varphi_i$ for $\varphi_i\in\Phi_i$, $\gamma_i\in G_i$.\label{thrm_pres_autg_r3}
            \item $\alpha_{ij}^{\left(\gamma_i\right)}\alpha_{k\ell}^{\left(\gamma_k\right)}=\alpha_{k\ell}^{\left(\gamma_k\right)}\alpha_{ij}^{\left(\gamma_i\right)}$ for $\gamma_i\in G_i$, $\gamma_k\in G_k$, $k\neq j$ and $\ell\neq i,j$.\label{thrm_pres_autg_r4}
            \item $\alpha_{ki}^{\left(\gamma_k\right)}\alpha_{ij}^{\left(\gamma_i\right)}\left(\alpha_{ki}^{\left(\gamma_k\right)}\right)^{-1}=\left(\alpha_{kj}^{\left(\gamma_k\right)}\right)^{-1}\alpha_{ij}^{\left(\gamma_i\right)}\alpha_{kj}^{\left(\gamma_k\right)}$ for $\gamma_i\in G_i$, $\gamma_k\in G_k$.\label{thrm_pres_autg_r5}\footnote{A misprint in \cite{fr1} was corrected in \cite[\S{} 4]{gilbert}.}
            \item $\omega_{ij}\varphi_i=\varphi_j\omega_{ij}$ if $\varphi_i\in\Phi_i$ and $\varphi_j\in\Phi_j$ are related by $\left(\omega_{ij}\varphi_i\right)_{\left|G_i\right.}=\left(\varphi_j\omega_{ij}\right)_{\left|G_i\right.}$.\label{thrm_pres_autg_r6}
            \item $\omega_{ij}\varphi_k=\varphi_k\omega_{ij}$ for $k\neq i,j$.\label{thrm_pres_autg_r7}
            \item $\omega_{ij}\alpha_{ij}^{\left(\gamma_i\right)}=\alpha_{ji}^{\left(\omega_{ij}\gamma_i\right)}\omega_{ij}$ for $\gamma_i\in G_i$.\label{thrm_pres_autg_r8}
            \item $\omega_{ij}\alpha_{ik}^{\left(\gamma_i\right)}=\alpha_{jk}^{\left(\omega_{ij}\gamma_i\right)}\omega_{ij}$ for $\gamma_i\in G_i$.\label{thrm_pres_autg_r9}
            \item $\omega_{ij}\alpha_{ki}^{\left(\gamma_k\right)}=\alpha_{kj}^{\left(\gamma_k\right)}\omega_{ij}$ for $\gamma_k\in G_k$.\label{thrm_pres_autg_r10}
            \item $\omega_{ij}\alpha_{k\ell}^{\left(\gamma_k\right)}=\alpha_{k\ell}^{\left(\gamma_k\right)}\omega_{ij}$ for $\gamma_k\in G_k$, $k,\ell\neq i,j$.\label{thrm_pres_autg_r11}
        \end{enumerate}
    \end{thrm}
    
    \paragraph{}
    To obtain a presentation of $\Out(G)$, it suffices to add relations whose normal closure is the subgroup $\Inn(G)$ of inner automorphisms. To write the presentation, we identify generators of $\Aut(G)$ with their image in $\Out(G)$.
    
    \begin{cons}\label{thrm_pres_outg}
        Let $G=G_1\ast\cdots\ast G_n$ be a free product of groups that are freely indecomposable and not isomorphic to $\mathbb{Z}$. Then $\Out(G)$ is generated by factor automorphisms, permutation automorphisms and Dehn twists. Moreover, relations $\ref{thrm_pres_autg_r1}-\ref{thrm_pres_autg_r11}$ of Theorem \ref{thrm_pres_autg}, together with relation \ref{cons_pres_outg_12} below, give a presentation of $\Out(G)$:
        \begin{enumerate}[label=\textnormal{(\arabic*)}]
            \setcounter{enumi}{-1}
            \item $\varphi_i\prod_{j\neq i}\alpha_{ij}^{\left(\gamma_i\right)}=1$ for $\gamma_i\in G_i$, with $\varphi_i\in\Phi_i$ given by ${\varphi_i}_{\left|G_i\right.}=\ad\left(\gamma_i^{-1}\right)$.\label{cons_pres_outg_12}\qedhere\qed
        \end{enumerate}
    \end{cons}
    
    \subsection{Construction of the splitting}\label{subsec_relev_eh}
    
    \paragraph{}
    Using the above presentation of $\Out(G)$, we are ready to split the projection $\Aut(G)/N\twoheadrightarrow\Out(G)$.
    
    Let $S$ denote the generating set of $\Out(G)$ introduced above. Observe that each generator $s\in S$ was defined by specifying a representative, which we will write $s_0$, of $s$ in $\Aut(G)$. We write $R\subset F(S)$ for the set of relations of $\Aut(G)$ and $R_1$ for the extra relations of $\Out(G)$. We thus have $\Out(G)=\left<S\mid R\cup R_1\right>$. Given a word $w$ in the free group $F(S)$, we denote by $w_0\in F\left(S_0\right)$ the corresponding word obtained after substituting $s\mapsto s_0$ (with $S_0=\left\{s_0,\:s\in S\right\}$).
    
    We aim to define a splitting of the projection $\Aut(G)/N\twoheadrightarrow\Out(G)$. We need to associate an element $\hat{s}\in\Aut(G)$ with each generator $s\in S$ in such a way that\[\hat{s}^{-1}s_0\in\Inn(G),\]and that for all $r\in R\cup R_1$, the lift $\hat{r}$ of the relation $r$ to $\Aut(G)$ satisfies\[\hat{r}\in\Inn(N).\]If we take $\hat{s}=s_0$ for all $s$, the relations $r\in R$ which hold in $\Aut(G)$ will automatically satisfy $\hat{r}=r_0=1$, but the relations of $R_1$ will not lift trivially.
    
    We thus need to correct the lifts $\hat{s}$ of the different generators $s\in S$ in order to satisfy all the relations of $\Out(G)$. When $G$ is a free group, Bridson and Vogtmann \cite{bv} correct the lifts of the generators corresponding to transvections by a translation along an axis; in other words, they make a component of the Galois group (which splits as a direct product) act on the lifts. This amounts to composing the lifts of some generators by a conjugation. This idea leads to the following proof.
    
    \begin{prop}\label{prop_scindement_autgn_outg}
        Let $G=G_1\ast\cdots\ast G_n$ be a free product of abelian groups that are not isomorphic to $\mathbb{Z}$. For each $i\in\left\{1,\dots,n\right\}$, we pick an integer $r_i\in\mathbb{Z}$ coprime to $n-1$, so that $r_i=r_j$ as soon as $G_i\cong G_j$. We consider the characteristic subgroup $N=G'G_1^{r_1}\cdots G_n^{r_n}$ of $G$.
        
        Then the projection $\Aut(G)/N\twoheadrightarrow\Out(G)$ splits.
    \end{prop}
    \begin{proof}
        For all $i\in\left\{1,\dots,n\right\}$, since $r_i$ is coprime to $n-1$, we may pick integers $u_i,t_i\in\mathbb{Z}$ such that $u_i(n-1)+t_ir_i=1$. We then lift the generators of $\Out(G)$ as follows:
        \begin{itemize}
            \item For a factor automorphism $\varphi_i\in\Phi_i$, we take the standard lift: $\hat{\varphi}_i=\left(\varphi_i\right)_0$.
            \item For a permutation automorphism $\omega_{ij}\in\Omega$, we take the standard lift: $\hat{\omega}_{ij}=\left(\omega_{ij}\right)_0$.
            \item For a Dehn twist $\alpha_{ij}^{\left(\gamma_i\right)}\in A_{ij}$, we correct the standard lift by\[\hat{\alpha}_{ij}^{\left(\gamma_i\right)}=\ad\left(\gamma_i^{u_i}\right)\circ\left(\alpha_{ij}^{\left(\gamma_i\right)}\right)_0.\]
        \end{itemize}
        We have to check that the images of $\hat{\varphi}_i$, $\hat{\omega}_{ij}$ and $\hat{\varphi}_{ij}^{\left(\gamma_i\right)}$ in $\Aut(G)/N$ satisfy the relations $\ref{cons_pres_outg_12}-\ref{thrm_pres_autg_r11}$ of $\Out(G)$.
        
        Since the lifts of factor automorphisms and permutation automorphisms are not corrected, the relations involving these types of automorphisms only are automatically satisfied. This is the case of relations \ref{thrm_pres_autg_r1}, \ref{thrm_pres_autg_r6} and \ref{thrm_pres_autg_r7}. There remain four kinds of relations:
        \begin{description}
            \item[$\ref{thrm_pres_autg_r2}-\ref{thrm_pres_autg_r3}$] The relations between factor automorphisms and Dehn twists. We check by computation that they are satisfied in $\Aut(G)$, and therefore in $\Aut(G)/N$.
            \item[$\ref{thrm_pres_autg_r8}-\ref{thrm_pres_autg_r11}$] The relations between permutation automorphisms and Dehn twists. We check by computation that they are satisfied in $\Aut(G)$, and therefore in $\Aut(G)/N$.
            \item[$\ref{thrm_pres_autg_r4}-\ref{thrm_pres_autg_r5}$] The relations between Dehn twists. We check that
            \begin{align*}
                \hat{\alpha}_{ij}^{\left(\gamma_i\right)}\hat{\alpha}_{k\ell}^{\left(\gamma_k\right)}&=\ad\left(\left[\gamma_i^{u_i},\gamma_k^{u_k}\right]\right)\circ\hat{\alpha}_{k\ell}^{\left(\gamma_k\right)}\hat{\alpha}_{ij}^{\left(\gamma_i\right)},\\
                \hat{\alpha}_{ki}^{\left(\gamma_k\right)}\hat{\alpha}_{ij}^{\left(\gamma_i\right)}\left(\hat{\alpha}_{ki}^{\left(\gamma_k\right)}\right)^{-1}&=\ad\left(\left[\gamma_k^{u_k-1},\gamma_i^{u_i}\right]\left[\gamma_i^{u_i},\gamma_k^{-u_k}\right]\right)\circ\left(\hat{\alpha}_{kj}^{\left(\gamma_k\right)}\right)^{-1}\hat{\alpha}_{ij}^{\left(\gamma_i\right)}\hat{\alpha}_{kj}^{\left(\gamma_k\right)},
            \end{align*}
            so these relations are satisfied in $\Aut(G)/N$ (but not in $\Aut(G)$).
            \item[$\ref{cons_pres_outg_12}$] The relation of $\Out(G)$. If $\varphi_i\in\Phi_i$ is given by ${\varphi_i}_{\left|G_i\right.}=\ad\left(\gamma_i^{-1}\right)$, we see that\[\hat{\varphi}_i\prod_{j\neq i}\hat{\alpha}_{ij}^{\left(\gamma_i\right)}=\ad\left(\gamma_i^{u_i(n-1)-1}\right)\circ\hat{\varphi}_i=\ad\left(\gamma_i^{t_ir_i}\right)\circ\hat{\varphi}_i.\]Since $G_i$ is abelian, $\hat{\varphi}_i=\id_G$, and since $G_i^{r_i}\subset N$, we have $\gamma_i^{t_ir_i}\in N$, so the last relation is satisfied in $\Aut(G)/N$.
        \end{description}
        Hence, the map defined on the generators of $\Out(G)$ by $s\mapsto\hat{s}$ induces a group homomorphism $\Out(G)\rightarrow\Aut(G)/N$, which is a splitting of the projection since each generator $s$ is the image of $\hat{s}$ in $\Out(G)$.
    \end{proof}
    
    \subsection{Embedding \texorpdfstring{$\Out(G)$}{Out(G)} into \texorpdfstring{$\Out(N)$}{Out(N)}}
    
    \paragraph{}
    Under certain assumptions, we have constructed a monomorphism $\Out(G)\hookrightarrow\Aut(G)/N$ by splitting the projection $\Aut(G)/N\twoheadrightarrow\Out(G)$. Since $N$ is a characteristic subgroup of $G$, we have in addition a restriction homomorphism $\Aut(G)\rightarrow\Aut(N)$ given by $\psi\mapsto\psi_{\left|N\right.}$, which induces a morphism $\Aut(G)/N\rightarrow\Out(N)$. The following lemma generalises a result of Collins \cite{collins} on the injectivity of the restriction homomorphism. We adapt an argument given by Bridson and Vogtmann \cite[Lem. 1]{bv} in the case where $G$ is a free group.
    
    \begin{lmm}\label{lmm_collins}
        Let $G=G_1\ast\cdots\ast G_n$ be a free product of non-trivial groups, with $n\geq 2$. We assume that $G\not\cong\mathbb{Z}/2\ast\mathbb{Z}/2$. If $N$ is a characteristic finite index subgroup of $G$, then the restriction homomorphism $\Aut(G)\rightarrow\Aut(N)$ is injective.
    \end{lmm}
    \begin{proof}
        Let $\psi\in\Aut(G)$ such that $\psi_{\left|N\right.}=\id_N$. We write $k=\left[G:N\right]$. Hence, for all $x\in G$, since $x^k\in N$, we have\[\psi(x)^k=\psi\left(x^k\right)=x^k.\]If $x$ does not belong to a conjugate of one of the factors $G_i$, then considering the normal forms of $x$ and $\psi(x)$ shows that $\psi(x)=x$.
        
        Let us now observe that, for $g_i\in G_i\smallsetminus\{1\}$ and $g_j\in G_j\smallsetminus\{1\}$, with $i\neq j$, we have\[\psi\left(g_i\right)\psi\left(g_j\right)^{-1}=\psi\left(g_ig_j^{-1}\right)=g_ig_j^{-1},\]hence $g_i^{-1}\psi\left(g_i\right)=g_j^{-1}\psi\left(g_j\right)$. Hence, there is an element $\gamma\in G$ such that for all $x\in\bigcup_iG_i\smallsetminus\{1\}$,\[\psi(x)=x\gamma.\]For $g_i\in G_i\smallsetminus\{1\}$ and $g_j\in G_j\smallsetminus\{1\}$, with $i\neq j$, we have in addition\[g_j\gamma g_i\gamma=\psi\left(g_j\right)\psi\left(g_i\right)=\psi\left(g_jg_i\right)=g_jg_i,\]so that $g_i=\gamma g_i\gamma$.
        
        If $\gamma\neq 1$, we write $\gamma$ in normal form as\[\gamma=w_{i_1}\cdots w_{i_r},\]with $r\geq1$, $i_1,\dots,i_r\in\left\{1,\dots,n\right\}$, $i_\ell\neq i_{\ell+1}$ and $w_{i_\ell}\in G_{i_\ell}\smallsetminus\{1\}$. We thus have, for all $x\in\bigcup_iG_i\smallsetminus\{1\}$,
        \begin{equation*}
            x=w_{i_1}\cdots w_{i_r}xw_{i_1}\cdots w_{i_r}.\tag{$\ast$}
        \end{equation*}
        If $i_1=i_r$, then we can take $j\neq i_1$ (because $n\geq 2$), $x\in G_j\smallsetminus\{1\}$, so that the right-hand side of $\left(\ast\right)$ is reduced, which is impossible. Hence, $i_1\neq i_r$. Since the right-hand side of $\left(\ast\right)$ is not reduced, any choice of $x\in\bigcup_iG_i\smallsetminus\{1\}$ satisfies $x\in\left\{w_{i_r}^{-1},w_{i_1}^{-1}\right\}$. This is impossible because $\sum_i\left(\left|G_i\right|-1\right)\geq 3$ since $G\not\cong\mathbb{Z}/2\ast\mathbb{Z}/2$ and $n\geq 2$.
        
        Therefore $\gamma=1$, so $\psi_{\left|G_i\right.}=\id_{G_i}$ for all $i$ and $\psi=\id_G$.
    \end{proof}
    
    Combining this lemma to Proposition \ref{prop_scindement_autgn_outg} (which gives an embedding $\Out(G)\hookrightarrow\Aut(G)/N$), we obtain our first embedding result.
    
    \begin{mthrm}\label{thrm_plong_outg}
        Let $G=G_1\ast\cdots\ast G_n$ be a free product of abelian groups that are not isomorphic to $\mathbb{Z}$, with $n\geq 2$. For each $i\in\left\{1,\dots,n\right\}$, we choose an integer $r_i\in\mathbb{Z}$ coprime to $n-1$, in such a way that $r_i=r_j$ as soon as $G_i\cong G_j$. We consider the characteristic subgroup $N=G'G_1^{r_1}\cdots G_n^{r_n}$ of $G$. Then there is an embedding\[\Out(G)\hookrightarrow\Out(N).\]
    \end{mthrm}
    \begin{proof}
        Proposition \ref{prop_scindement_autgn_outg} gives an embedding $\Out(G)\hookrightarrow\Aut(G)/N$. Lemma \ref{lmm_collins} implies in addition that $\Aut(G)\rightarrow\Aut\left(N\right)$ is injective if $G\not\cong\mathbb{Z}/2\ast\mathbb{Z}/2$; it follows that $\Aut(G)/N\rightarrow\Out(N)$ is injective. By composition, we get an embedding $\Out(G)\hookrightarrow\Aut(G)/N\hookrightarrow\Out(N)$.
        
        If $G\cong\mathbb{Z}/2\ast\mathbb{Z}/2$, we have $\Out(G)\cong\mathbb{Z}/2$. Moreover, $N=G'\cong\mathbb{Z}$ (if $r_1$ is even) or  $N=G$ (if $r_1$ is odd); in both cases, $\Out(N)\cong\mathbb{Z}/2$, so there is an embedding $\Out(G)\hookrightarrow\Out(N)$.
    \end{proof}
    
    If the factors are finite and coprime to $n-1$, we can take $r_i=\left|G_i\right|$ for all $i$, so that $N=G'$ is a free subgroup of $G$ (by Lemma \ref{lmm_sousgp_deriv_libre}). This yields the following corollary.
    
    \begin{cons}\label{cons_rep_libre_sans_facteur_libre}
        Let $G=G_1\ast\cdots\ast G_n$ be a free product of finite abelian groups, with $n\geq 2$. We assume that $n-1$ is coprime to the order $\left|G_i\right|$ of each factor $G_i$. Then there exists a free subgroup $F$ of finite rank and of finite index in $G$ such that there is an embedding\[\Out(G)\hookrightarrow\Out(F).\]In particular, $\Out(G)$ has a faithful free representation.\qed
    \end{cons}
    
\section{Proof of Theorem \ref{thrm_plong_outg_2}}\label{sect_second_emb_result}
    
    \subsection{Generators and relations with a free factor}\label{subsec_gen_rel_facteur_libre}
        
    \paragraph{}
    We now turn our attention to the case where $G=G_1\ast\cdots\ast G_n$ is a free product of freely indecomposable groups and infinite cyclic groups. We order the factors in such a way that $G_1\cong\cdots\cong G_d\cong\mathbb{Z}$, and $G_{d+1},\dots,G_n\not\cong\mathbb{Z}$. In other words,\[G=F_d\ast G_{d+1}\ast\cdots\ast G_n,\]with $F_d$ free of rank $d$.
    
    A major difference with the case $d=0$ is that, as shown by Grushko's Theorem, there is now significantly more flexibility in the free product decomposition of $G$. In particular, we cannot construct a characteristic subgroup of $G$ as in Lemma \ref{lmm_gpg1r1gnrn_caract}. Instead, we will work with the characteristic subgroup $N=G'G^r$ for $r\in\mathbb{Z}$.
    
    \paragraph{}
    Nielsen \cite{nielsen}, and more recently Gersten \cite{gersten}, gave presentations of $\Aut\left(F_d\right)$, Fuchs-Rabinovich \cite{fr1} did the same for $\Aut\left(G_{d+1}\ast\cdots\ast G_n\right)$, as we have seen above. In fact, in a second article \cite{fr2}, Fuchs-Rabinovich gives a complete presentation of $\Aut(G)$. The resulting presentation is rather long, but it generalises the results for the free part and the part without free factor in a relatively natural way.
    
    The generators of Fuchs-Rabinovich for $\Aut(G)$ are the following.
    \begin{nota}\label{nota_gen_fr2}
        Let $G=G_{1}\ast\cdots\ast G_d\ast G_{d+1}\ast\cdots\ast G_n$, with $G_1,\dots,G_d$ isomorphic to $\mathbb{Z}$ and equipped with respective generators $a_1,\dots,a_d$, and $G_{d+1},\dots,G_n$ freely indecomposable not isomorphic to $\mathbb{Z}$.
        \begin{description}
            \item[\textnormal{Factor automorphisms:}] For $i\in\left\{d+1,\dots,n\right\}$, we consider the group $\Phi_i$ of automorphisms of the factor $G_i$ (see Notation \ref{nota_gen_fr}).
            \item[\textnormal{Reflections:}] For $i\in\left\{1,\dots,d\right\}$, we set $\tau_i:a_i\mapsto a_i^{-1}$. The automorphism $\tau_i$ is called the \emph{reflection} of $G_i$ and we have $\Phi_i=\left\{\id,\tau_i\right\}$.
            \item[\textnormal{Permutation automorphisms:}] We consider $\Omega=\left\{\omega_{ij},\:i\neq j,\:G_i\cong G_j\right\}$ the group of permutation automorphisms, defined as in Notation \ref{nota_gen_fr}, with the additional convention that $\omega_{ij}:a_i\mapsto a_j$ and $a_j\mapsto a_i$ when $i,j\leq d$.
            \item[\textnormal{Dehn twists:}] For $i\in\left\{1,\dots,n\right\}$, $j\in\left\{d+1,\dots,n\right\}$, $i\neq j$, we consider $A_{ij}$ the group of Dehn twists of type $ij$ (see Notation \ref{nota_gen_fr}). If $i\in\left\{1,\dots,d\right\}$, we write $\alpha_{ij}=\alpha_{ij}^{\left(a_i\right)}$.
            \item[\textnormal{Right transvections:}] For $i\in\left\{1,\dots,n\right\}$, $j\in\left\{1,\dots,d\right\}$, $i\neq j$, $\gamma_i\in G_i$, we define
            \begin{align*}
                \rho_{ij}^{\left(\gamma_i\right)}:a_j\in G_j&\mapsto a_j\gamma_i,\\
                g_k\in G_k&\mapsto g_k\quad\textrm{if $k\neq j$}.
            \end{align*}
            We write $P_{ij}=\left\{\rho_{ij}^{\left(\gamma_i\right)},\:\gamma_i\in G_i\right\}$ and $P=\bigcup_{i\neq j}P_{ij}$. Elements of $P$ are called \emph{right transvections}. If $i\in\left\{1,\dots,d\right\}$, we write $\rho_{ij}=\rho_{ij}^{\left(a_i\right)}$.
            \item[\textnormal{Left transvections:}] For $i\in\left\{1,\dots,n\right\}$, $j\in\left\{1,\dots,d\right\}$, $i\neq j$, $\gamma_i\in G_i$, we define
            \begin{align*}
                \lambda_{ij}^{\left(\gamma_i\right)}:a_j\in G_j&\mapsto \gamma_ia_j,\\
                g_k\in G_k&\mapsto g_k\quad\textrm{if $k\neq j$}.
            \end{align*}
            We write $\Lambda_{ij}=\left\{\lambda_{ij}^{\left(\gamma_i\right)},\:\gamma_i\in G_i\right\}$ and $\Lambda=\bigcup_{i\neq j}\Lambda_{ij}$. Elements of $\Lambda$ are called \emph{left transvections}. If $i\in\left\{1,\dots,d\right\}$, we write $\lambda_{ij}=\lambda_{ij}^{\left(a_i\right)}$.
        \end{description}
    \end{nota}
    Note that we have inverted Bridson and Vogtmann's notations: we write $\rho_{ij},\lambda_{ij}$ instead of $\rho_{ji},\lambda_{ji}$ to be consistent with Fuchs-Rabinovich.
    
    We have a presentation of $\Aut(G)$ as follows.
    \begin{thrm}[Fuchs-Rabinovich \cite{fr2}]\label{thrm_pres_autg_2}
        Let $G=G_1\ast\cdots\ast G_d\ast G_{d+1}\ast\cdots\ast G_n$ be a free product of freely indecomposable groups, with $G_1\cong\cdots\cong G_d\cong\mathbb{Z}$ and $G_{d+1},\dots,G_n\not\cong\mathbb{Z}$. Then $\Aut(G)$ is generated by factor automorphisms, reflections, permutation automorphisms, Dehn twists, left and right transvections. Moreover, the following relations give a presentation of $\Aut(G)$:
        \begin{description}
            \item[$\ref{thrm_pres_autg_r1}-\ref{thrm_pres_autg_r11}$] The relations of Theorem \ref{thrm_pres_autg}, written for all the indices for which they make sense.
        \end{description}
        \begin{enumerate}[label=\textnormal{(\arabic*)}]
            \setcounter{enumi}{11}
            \item Nielsen's relations \cite{nielsen} of $\Aut\left(F_d\right)$, written for all the indices for which they make sense.\label{thrm_pres_autg_r12}
            \item $\alpha_{ij}^{\left(\gamma_i\right)}\tau_k=\tau_k\alpha_{ij}^{\left(\gamma_i\right)}$ for $\gamma_i\in G_i$, $i,j\neq k$.\label{thrm_pres_autg_r13}
            \item $\alpha_{ij}\tau_i=\tau_i\alpha_{ij}^{-1}$.\label{thrm_pres_autg_r14}
            \item $\rho_{ij}^{\left(\gamma_i\right)}\varphi_k=\varphi_k\rho_{ij}^{\left(\gamma_i\right)}$ for $\gamma_i\in G_i$, $\varphi_k\in\Phi_k$, $i,j\neq k$.\label{thrm_pres_autg_r15}
            \item $\rho_{ij}^{\left(\varphi_i\gamma_i\right)}\varphi_i=\varphi_i\rho_{ij}^{\left(\gamma_i\right)}$ for $\gamma_i\in G_i$, $\varphi_i\in\Phi_i$.\label{thrm_pres_autg_r16}
            \item $\rho_{ij}^{\left(\gamma_i\right)}\alpha_{k\ell}^{\left(\gamma_k\right)}=\alpha_{k\ell}^{\left(\gamma_k\right)}\rho_{ij}^{\left(\gamma_i\right)}$ for $\gamma_i\in G_i$, $\gamma_k\in G_k$, $i,j\neq k,\ell$.\label{thrm_pres_autg_r17}
            \item $\rho_{ij}^{\left(\gamma_i\right)}\alpha_{i\ell}^{\left(\gamma_i\right)}=\alpha_{i\ell}^{\left(\gamma_i\right)}\rho_{ij}^{\left(\gamma_i\right)}$ for $\gamma_i\in G_i$.\label{thrm_pres_autg_r18}
            \item $\rho_{ij}^{\left(\gamma_i\right)}\alpha_{j\ell}=\alpha_{j\ell}\rho_{ij}^{\left(\gamma_i\right)}\alpha_{i\ell}^{\left(\gamma_i\right)}$ for $\gamma_i\in G_i$.\footnote{A misprint in \cite{fr2} was corrected in \cite[\S{} 4]{gilbert}.}\label{thrm_pres_autg_r19}
            \item $\left(\rho_{kj}^{\left(\gamma_k\right)}\right)^{-1}\rho_{ij}^{\left(\gamma_i\right)}\rho_{kj}^{\left(\gamma_k\right)}\alpha_{ki}^{\left(\gamma_k\right)}=\alpha_{ki}^{\left(\gamma_k\right)}\rho_{ij}^{\left(\gamma_i\right)}$ for $\gamma_i\in G_i$, $\gamma_k\in G_k$.\label{thrm_pres_autg_r20}
            \item $\alpha_{ji}\rho_{ij}^{\left(\gamma_i\right)}=\lambda_{ij}^{\left(\gamma_i\right)}\alpha_{ji}\varphi_i$ for $\gamma_i\in G_i$ with $\left(\varphi_i\right)_{\left|G_i\right.}=\ad\left(\gamma_i^{-1}\right)$.\label{thrm_pres_autg_r21}
            \item $\tau_k\varphi_i=\varphi_i\tau_k$ for $\varphi_i\in\Phi_i$, $i\neq k$.\label{thrm_pres_autg_r22}
            \item $\tau_i\omega_{k\ell}=\omega_{k\ell}\tau_i$ for $i\neq k,\ell$.\label{thrm_pres_autg_r23}
        \end{enumerate}
    \end{thrm}
    \begin{rk}
        In Theorem \ref{thrm_pres_autg_2}, relations $\ref{thrm_pres_autg_r1}-\ref{thrm_pres_autg_r12}$ are more numerous than the union of relations of $\Aut\left(F_d\right)$ and $\Aut\left(G_{d+1}\ast\cdots\ast G_n\right)$: they include for example relations involving Dehn twists conjugating one of the factors $G_i$ (with $i>d$) by an element of $F_d$. However, the choice of notations means that they are written in the same way as in $\Aut\left(F_d\right)$ or in $\Aut\left(G_{d+1}\ast\cdots\ast G_n\right)$.
    \end{rk}
    
    As before, we can easily deduce a presentation of $\Out(G)$.
    \begin{cons}
        Let $G=G_1\ast\cdots\ast G_d\ast G_{d+1}\ast\cdots\ast G_n$ be a free product of freely indecomposable groups, with $G_1\cong\cdots\cong G_d\cong\mathbb{Z}$ and $G_{d+1},\dots,G_n\not\cong\mathbb{Z}$. Then $\Out(G)$ is generated by factor automorphisms, reflections, permutation automorphisms, Dehn twists, left and right transvections. Moreover, relations $\ref{thrm_pres_autg_r1}-\ref{thrm_pres_autg_r23}$ of Theorem \ref{thrm_pres_autg_2}, together with relation \ref{cons_pres_outg_0} below, give a presentation of $\Out(G)$:
        \begin{enumerate}[label=\textnormal{(\arabic*)}]
            \setcounter{enumi}{-1}
            \item $\varphi_i\circ\prod_{\substack{j\leq d\\j\neq i}}\rho_{ij}^{\left(\gamma_i\right)}\left(\lambda_{ij}^{\left(\gamma_i\right)}\right)^{(-1)}\circ\prod_{\substack{j>d\\j\neq i}}\alpha_{ij}^{\left(\gamma_i\right)}=1$ for $\gamma_i\in G_i$, with ${\varphi_i}_{\left|G_i\right.}=\ad\left(\gamma_i^{-1}\right)$.\label{cons_pres_outg_0}\qedhere\qed
        \end{enumerate}
    \end{cons}
    
    \subsection{Embedding results with a free factor}\label{subsec_plong_facteur_libre}
    
    \paragraph{}
    We now assume that $G=F_d\ast G_{d+1}\ast\cdots\ast G_n$, with $F_d$ free of rank $d$, $G_{d+1},\dots,G_n$ abelian (hence freely indecomposable) and not isomorphic to $\mathbb{Z}$. We consider the characteristic subgroup $N=G'G^r$, where $r\in\mathbb{Z}$ is an integer to be fixed. To construct a splitting of $\Aut(G)/N\twoheadrightarrow\Out(G)$, we follow the same method as in \ref{subsec_relev_eh}: for each generator $s\in\Out(G)$, we pick a correction $\hat{s}\in\Aut(G)/N$ of the standard lift $s_0\in\Aut(G)/N$, in such a way that the corrected lifts $\hat{s}$ satisfy the relations of $\Out(G)$.
    
    In fact, the generators of $\Out(G)$ in this setting are all of the same kind as those used by Bridson and Vogtmann \cite{bv} when $G=F_d$ (for transvections and reflections) or as those used in \ref{subsec_relev_eh} when $G=G_{d+1}\ast\cdots\ast G_n$ (for factor automorphisms, permutation automorphisms and Dehn twists). Hence, we can simply take the same corrections as before. Relations $\ref{thrm_pres_autg_r1}-\ref{thrm_pres_autg_r12}$ of Theorem \ref{thrm_pres_autg_2} will be automatically satisfied, and it will suffice to check $\ref{thrm_pres_autg_r13}-\ref{thrm_pres_autg_r23}$.
    
    \begin{prop}
        Let $G=F_d\ast G_{d+1}\ast\cdots\ast G_n$, with $F_d$ free of rank $d$, $G_{d+1},\dots,G_n$ abelian not isomorphic to $\mathbb{Z}$. We choose an integer $r\in\mathbb{Z}$ coprime to $n-1$ and we consider the characteristic subgroup $N=G'G^r$ of $G$.
        
        Then the projection $\Aut(G)/N\twoheadrightarrow\Out(G)$ splits.
    \end{prop}
    \begin{proof}
        We may pick integers $u,t\in\mathbb{Z}$ such that $u(n-1)+tr=1$. We then lift the generators of $\Out(G)$ as follows:
        \begin{itemize}
            \item For factor automorphisms, permutation automorphisms and Dehn twists, we take the same corrections as in the proof of Proposition \ref{prop_scindement_autgn_outg} (replacing $u_i$ by $u$).
            \item Right transvections are not corrected (we keep the standard lifts).
            \item Left transvections are corrected by $\hat{\lambda}_{ij}^{\left(\gamma_i\right)}=\ad\left(\gamma_i^{-u}\right)\circ\left(\lambda_{ij}^{\left(\gamma_i\right)}\right)_0$.
            \item Reflections are corrected by $\hat{\tau}_{i}=\ad\left(a_i^{u}\right)\circ\left(\tau_{i}\right)_0$.
        \end{itemize}
        Observe that these corrections of transvections and reflections are an algebraic restatement of the corrections of Bridson and Vogtmann \cite{bv}.
        
        It follows that relations $\ref{thrm_pres_autg_r1}-\ref{thrm_pres_autg_r12}$ of Theorem \ref{thrm_pres_autg_2} are satisfied: their verification would be a rewriting of the computations of Proposition \ref{prop_scindement_autgn_outg} (for $\ref{thrm_pres_autg_r1}-\ref{thrm_pres_autg_r11}$) and of Bridson and Vogtmann \cite{bv} (for $\ref{thrm_pres_autg_r12}$). Moreover, relations $\ref{thrm_pres_autg_r15}-\ref{thrm_pres_autg_r16}$ are satisfied because they only involve generators whose lifts are not corrected.
        
        We check that relations $\ref{thrm_pres_autg_r14}-\ref{thrm_pres_autg_r18}$, $\ref{thrm_pres_autg_r20}$ and $\ref{thrm_pres_autg_r22}-\ref{thrm_pres_autg_r23}$ are satisfied in $\Aut(G)$. We give the results of our computations for the other relations, and we see that they are satisfied in $\Aut(G)/N$:
        \begin{enumerate}[label=\textnormal{(\arabic*)}]
            \item[(13)] $\hat{\alpha}_{ij}^{\left(\gamma_i\right)}\hat{\tau}_k=\ad\left(\left[\gamma_i^u,a_k^{u}\right]\right)\circ\hat{\tau}_k\hat{\alpha}_{ij}^{\left(\gamma_i\right)}$.
            \item[(19)] $\hat{\rho}_{ij}^{\left(\gamma_i\right)}\hat{\alpha}_{j\ell}=\ad\left(\left(a_j\gamma_i\right)^u\gamma_i^{-u}a_j^{-u}\right)\circ\hat{\alpha}_{j\ell}\hat{\rho}_{ij}^{\left(\gamma_i\right)}\hat{\alpha}_{i\ell}^{\left(\gamma_i\right)}$.
            \item[(21)] $\hat{\alpha}_{ji}\hat{\rho}_{ij}^{\left(\gamma_i\right)}=\ad\left(a_j^{u}\left(\gamma_ia_j\right)^{-u}\gamma_i^u\right)\circ\hat{\lambda}_{ij}^{\left(\gamma_i\right)}\hat{\alpha}_{ji}\hat{\varphi}_i$.
            \item[(0)] $\hat{\varphi}_i\circ\prod_{\substack{j\leq d\\j\neq i}}\hat{\rho}_{ij}^{\left(\gamma_i\right)}\left(\hat{\lambda}_{ij}^{\left(\gamma_i\right)}\right)^{(-1)}\circ\prod_{\substack{j>d\\j\neq i}}\hat{\alpha}_{ij}^{\left(\gamma_i\right)}=\ad\left(\gamma_i^{u(n-1)-1}\right)=\ad\left(\gamma_i^{tr}\right)$.
        \end{enumerate}
        Hence, $s\mapsto\hat{s}$ gives a splitting of the morphism $\Aut(G)/N\twoheadrightarrow\Out(G)$.
    \end{proof}
    
    \paragraph{}
    Lemma \ref{lmm_collins}, which proved the injectivity of the restriction homomorphism $\Aut(G)\rightarrow\Aut(N)$, was independent of the presence of free factors in the free product. Hence, it remains valid in the present setting, which yields the following result with the same proof as Theorem \ref{thrm_plong_outg}.
    
    \begin{mthrm}\label{thrm_plong_outg_2}
        Let $G=F_d\ast G_{d+1}\ast\cdots\ast G_n$, with $n\geq2$, $F_d$ free of rank $d$, $G_{d+1},\dots,G_n$ abelian not isomorphic to $\mathbb{Z}$. We choose an integer $r\in\mathbb{Z}$ coprime to $n-1$ and we consider the characteristic subgroup $N=G'G^r$ of $G$. Then there is an embedding\[\pushQED{\qed}\Out(G)\hookrightarrow\Out(N).\qedhere\popQED\]
    \end{mthrm}
    \begin{rk}
        When $d=n$, we get the result of Bridson and Vogtmann \cite[Cor. A]{bv}.
    \end{rk}
    
    Taking $r=\left|G_{d+1}\right|\cdots\left|G_n\right|$ yields a generalisation of Corollary \ref{cons_rep_libre_sans_facteur_libre}.
    
    \begin{cons}\label{cons_rep_libre}
        Let $G=F_d\ast G_{d+1}\ast\cdots\ast G_n$, with $n\geq2$, $F_d$ free of rank $d$, $G_{d+1},\dots,G_n$ finite abelian. We assume that $n-1$ is coprime to the order $\left|G_i\right|$ of each factor $G_i$. Then there is a free subgroup $F$ of finite rank and of finite index in $G$ such that there is an embedding\[\Out(G)\hookrightarrow\Out(F).\]In particular, $\Out(G)$ has a faithful free representation.\qed
    \end{cons}
    
    Observe that, by increasing the rank of the free factor, we can always get to a situation where $n-1$ is coprime to $\left|G_i\right|$ for all $i$. We thus get our second corollary.
    
    \begin{cons}\label{cons_rep_libre_2}
        Let $G=F_d\ast G_{d+1}\ast\cdots\ast G_n$, with $F_d$ free of rank $d$, $G_{d+1},\dots,G_n$ finite abelian. Then there is an integer $k\geq0$ such that $\Out\left(F_k\ast G\right)$ has a faithful free representation.\qed
    \end{cons}
    
\section{Application to universal Coxeter groups}\label{sec_univ_coxeter}
    
    \paragraph{}
    Consider $W_n=\left<x_1,\dots,x_n\mid x_1^2,\dots,x_n^2\right>$. The group $W_n$ is called the \emph{universal Coxeter group} of rank $n$; it is the free product of $n$ copies of $\mathbb{Z}/2$, and it is certainly the simplest example of free product of freely indecomposable groups that are not isomorphic to $\mathbb{Z}$. It is therefore natural to ask what properties $\Out\left(W_n\right)$ shares with $\Out\left(F_n\right)$.
    
    We focus in particular on free representations of $\Out\left(W_n\right)$. For $\Aut\left(W_n\right)$, Mühlherr \cite{muehlherr} showed that there is an embedding $\Aut\left(W_n\right)\hookrightarrow\Aut\left(F_{n-1}\right)$ for all $n\geq 3$. On the contrary, Varghese \cite{varghese} showed that, if $n\geq 4$ and $m\leq n-2$, then any morphism from $\Aut\left(W_n\right)$ to $\Aut\left(F_m\right)$, $\Out\left(F_m\right)$ or $GL_m\left(\mathbb{Z}\right)$ has finite image.
    
    We turn our attention to the parallel question for outer automorphisms: we wish to know for which integers $n,m$ there are faithful free representations $\Out\left(W_n\right)\hookrightarrow\Out\left(F_m\right)$. As an application of the results of this text, we have the following corollary.
    \begin{cons}\label{cons_plong_outwn}
        If $n$ is even and $m=2^{n-1}(n-2)+1$, then there is an embedding\[\Out\left(W_n\right)\hookrightarrow\Out\left(F_m\right).\]
    \end{cons}
    \begin{proof}
        By Theorem \ref{thrm_plong_outg} (with $r_i=2$ for all $i$), there is an embedding $\Out\left(W_n\right)\hookrightarrow\Out\left(W_n'\right)$. Let us consider in addition the $W_n$-tree $A$ such that $W_n\backslash A$ is a tree of group with $n$ vertices labelled by $\mathbb{Z}/2$ and $n-1$ unlabelled edges. Lemma \ref{lmm_sousgp_deriv_libre} implies that $W_n'$ acts freely on $A$, so $W_n'\cong F_m$, where $m$ is the genus (i.e. the rank of the fundamental group) of the graph $W_n'\backslash A$. Observe that $W_n/W_n'\cong\left(\mathbb{Z}/2\right)^n$ and that in $A$, vertex stabilisers under $W_n/W_n'$ have order $2$ and edge stabilisers are trivial. Hence, each edge orbit under $W_n$ corresponds to $2^{n}$ edge orbits under $W_n'$, and each vertex orbit under $W_n$ corresponds to $2^{n-1}$ vertex orbits under $W_n'$. It follows that $W_n'\backslash A$ is a graph with $(n-1)2^n$ edges and $n2^{n-1}$ vertices, hence\[m=\left|E\left(W_n'\backslash A\right)\right|-\left|V\left(W_n'\backslash A\right)\right|+1=2^{n-1}(n-2)+1.\qedhere\]
    \end{proof}
    
    We have no reason to believe that this result covers all values of $n,m$ for which there are embeddings $\Out\left(W_n\right)\hookrightarrow\Out\left(F_m\right)$. To see this, let us examine the case $n=3$, on which Corollary \ref{cons_plong_outwn} does not say anything. The group $\Out\left(W_3\right)$ is simple enough to be computed explicitly, and we can construct a faithful free representation by hand.
    \begin{prop}
        \begin{enumerate}
            \item $\Out\left(W_3\right)\cong W_3\rtimes\mathfrak{S}_3$, where $\mathfrak{S}_3$ acts on $W_3$ by permutation of the generators.\label{prop_outw3_1}
            \item There is an embedding $\Out\left(W_3\right)\hookrightarrow\Out\left(F_4\right)$.
        \end{enumerate}
    \end{prop}
    \begin{proof}
        \begin{enumerate}
            \item We have a morphism $\Out\left(W_3\right)\rightarrow\mathfrak{S}_3$ given by the action of $W_3$ on conjugacy classes of generators; we denote by $\Out_0\left(W_3\right)$ its kernel. Hence, there is an exact sequence\[1\rightarrow\Out_0\left(W_3\right)\rightarrow\Out\left(W_3\right)\rightarrow\mathfrak{S}_3\rightarrow1.\]We can split this exact sequence by mapping $\left(i\,j\right)\mapsto\omega_{ij}$ (we follow Notation \ref{nota_gen_fr}). Therefore $\Out\left(W_3\right)\cong\Out_0\left(W_3\right)\rtimes\mathfrak{S}_3$. There remains to determine $\Out_0\left(W_3\right)$.
            
            It follows from Fuchs-Rabinovich's presentation of $\Out\left(W_3\right)$ (see Corollary \ref{thrm_pres_outg}) that the subgroup $\Out_0\left(W_3\right)$ is generated by Dehn twists (in this case, there is no factor automorphism), with relations $\ref{cons_pres_outg_12}$ and $\ref{thrm_pres_autg_r4}-\ref{thrm_pres_autg_r5}$ (they are the only ones that only involve Dehn twists). Each Dehn twist has order $2$. Writing $\alpha_{ij}=\alpha_{ij}^{\left(x_i\right)}$, relation $\ref{cons_pres_outg_12}$ yields\[\alpha_{12}=\alpha_{13},\quad\alpha_{21}=\alpha_{23},\quad\alpha_{31}=\alpha_{32}.\]Relations $\ref{thrm_pres_autg_r4}-\ref{thrm_pres_autg_r5}$ follow, which gives\[\Out_0\left(W_3\right)=\left<\alpha_{12},\alpha_{21},\alpha_{31}\mid\alpha_{12}^2,\alpha_{21}^2,\alpha_{31}^2\right>.\]We therefore see that $\Out_0\left(W_3\right)\cong W_3$, and the action of $\mathfrak{S}_3$ is given by $\left(1\,2\right)\cdot\alpha_{12}=\alpha_{21}$, $\left(1\,2\right)\cdot\alpha_{31}=\alpha_{32}=\alpha_{31}$ and $\left(2\,3\right)\cdot\alpha_{12}=\alpha_{13}=\alpha_{12}$, $\left(2\,3\right)\cdot\alpha_{31}=\alpha_{21}$. This corresponds to the action of $\mathfrak{S}_3$ by permutation of the generators of $W_3$.
            
            \item By \ref{prop_outw3_1}, it suffices to construct an embedding $W_3\rtimes\mathfrak{S}_3\hookrightarrow\Out\left(F_4\right)$. We denote by $x_1,x_2,x_3$ the generators of $W_3$, $a_1,\dots,a_4$ the generators of $F_4$. We write $\alpha_{ij}=\rho_{ij}\lambda_{ij}^{-1}$ for the Dehn twists in $F_4$ and $\tau_i$ for the reflection $a_i\mapsto a_i^{-1}$ (see Notation \ref{nota_gen_fr2}). For $\sigma\in\mathfrak{S}_3$, we define in addition $\omega_\sigma\in\Out\left(F_4\right)$ by $a_i\mapsto a_{\sigma(i)}$ for $i\in\left\{1,2,3\right\}$ and $a_4\mapsto a_4$.
            
            We then define $f:W_3\rtimes\mathfrak{S}_3\rightarrow\Out\left(F_4\right)$ by
            \begin{align*}
                x_i\in W_3&\longmapsto\alpha_{i4}\tau_i=\tau_i\alpha_{i4}^{-1}\in\Out\left(F_4\right),\\
                \sigma\in\mathfrak{S}_3&\longmapsto\omega_{\sigma}\in\Out\left(F_4\right).
            \end{align*}
            We have $f\left(x_i\right)^2=1$ and $f\left(\sigma\right)f\left(x_i\right)f\left(\sigma\right)^{-1}=f\left(x_{\sigma(i)}\right)$ for all $i\in\left\{1,2,3\right\}$ and $\sigma\in\mathfrak{S}_3$, so $f$ gives a well-defined group homomorphism $W_3\rtimes\mathfrak{S}_3\rightarrow\Out\left(F_4\right)$.
            
            We want to show that $f$ is injective. Let $\gamma=w\sigma\in\Ker f$, with $w\in W_3$ and $\sigma\in\mathfrak{S}_3$. Since $f\left(\gamma\right)=1$ in $\Out\left(F_4\right)$, $f(\gamma)$ acts trivially on the conjugacy classes of the generators $a_i$, so $\sigma=1$. We write $w=x_{i_1}\cdots x_{i_k}$ in normal form, with $k\geq 0$, $i_j\in\left\{1,2,3\right\}$, $i_{j+1}\neq i_j$. We then have
            \begin{align*}
                f\left(\gamma\right)\cdot a_i&=a_i^{\pm1}\quad\textrm{if $i\in\left\{1,2,3\right\}$}\\
                f\left(\gamma\right)\cdot a_4&=\ad\left(a_{i_k}^{\varepsilon_k}\cdots a_{i_1}^{\varepsilon_1}\right)\cdot a_4,
            \end{align*}
            with $\varepsilon_1,\dots,\varepsilon_k\in\left\{\pm1\right\}$. Since $f\left(\gamma\right)=1$ in $\Out\left(F_4\right)$, it follows that $a_{i_k}^{\varepsilon_k}\cdots a_{i_1}^{\varepsilon_1}=1$. Since this is a reduced form in $F_4$, we have $k=0$, so $w=1$ and $\gamma=1$.
            
            Hence, $f:W_3\rtimes\mathfrak{S}_3\hookrightarrow\Out\left(F_4\right)$ is an embedding.\qedhere
        \end{enumerate}
    \end{proof}
    
    \paragraph{}
    This example confirms that out method is not exhaustive: we have been able to construct embeddings $\Out\left(W_n\right)\hookrightarrow\Out\left(F_m\right)$ for some integers $n,m$, but there may exist other values of $n$ for which $\Out\left(W_n\right)$ admits faithful free representations.
    
    In general, the group $\Out(G)$ seems very likely to have faithful free representations for many free products $G$ for which we cannot say anything with the method used here.

\bibliography{2021-M2-Stage-Refs.bib}
\bibliographystyle{plain}
    
\end{document}